\documentclass{article}
\usepackage[english]{babel}
\usepackage[latin1]{inputenc}
\usepackage[T1]{fontenc}
\usepackage{amsmath,amssymb,amsthm}

\RequirePackage{ifpdf}
\ifpdf
   \usepackage[pdftex]{hyperref}
\else
   \usepackage[hypertex]{hyperref}
   \usepackage{srcltx}
\fi

\title{Sharply $2$-transitive groups}
\author{Katrin Tent and Martin Ziegler}
\date{24.8.2014}

\newtheorem{theorem}{Theorem}[section]

\newtheorem{lemma}[theorem]{Lemma}

\newtheorem{corollary}[theorem]{Corollary}
\newtheorem{definition}[theorem]{Definition}

\newcommand{\nc}{\newcommand}
\nc{\inv}{^{-1}}

\begin{document}
\maketitle
\begin{abstract}
  We give an explicit construction of sharply
  $2$-transitive groups with fixed point free involutions and
  without nontrivial abelian normal subgroup.
\end{abstract}

\section{Introduction}

The finite sharply $2$-transitive groups were classified by Zassenhaus
\cite{Z} in the 1930's. They were shown to always contain a regular
abelian normal subgroup. It remained an open question whether the same
holds for infinite sharply $2$-transitive groups. The first examples
of sharply $2$-transitive groups without abelian normal subgroup were
recently constructed in \cite{RST}. In these examples involutions have
no fixed points. We here give an alternative approach to such a
construction by using partially defined group actions.

\section{The construction}\label{S:terminology}

\begin{theorem}\label{t:main}
  Let $G_0$ be a group containing an involution $t$. Suppose that
  $G_0$ acts on a set $X$ and satisfies the following:
  \begin{enumerate}
  \item no nontrivial element of $G_0$ fixes more than one element
    of $X$ (we say that $G_0$ is $2$-sharp);
  \item all involutions are conjugate to $t$;
  \item $t$ does not fix any element of $X$.
  \end{enumerate}
  Then we can extend $G_0$ to a sharply $2$-transitive action of
  \[G=\bigl(G_0\ast_{\langle t\rangle}
  (\langle t\rangle\times F(S))\bigr)\ast F(R)\]
  on a suitable set $Y\supset X$, where $F(R), F(S)$ are free groups
  on disjoint sets $R, S$ with $|R|,|S|= \max{|G_0|,\aleph_0}$.
\end{theorem}
Note that $G$ does not contain any nontrivial abelian normal subgroup.
Hence we obtain:

\begin{corollary}\label{c:main}
Any group can be extended to a group acting sharply $2$-transitively
on some appropriate set without nontrivial abelian normal subgroup.
\end{corollary}

\begin{proof} By adding a direct factor of order $2$ if necessary
and iterated HNN-extensions
any group can be extended to a group with a unique nontrivial conjugacy
class of involutions. Letting this group act regularly on itself by right translation all assumptions of Theorem~\ref{t:main} are satisfied.
\end{proof}
\begin{definition}
  A partial action of $G$ on a set $X$ consists of an action of $G_0$
  on $X$ and (injective) partial actions of the generators in $S\cup
  R$ such that for $s\in S, x\in X$ if $xs$ is defined, then so is
  $(xt)s$ and we have $(xt)s=(xs)t$.
\end{definition}

Any element of $G$ can be written as a reduced word in elements of
\[\mathcal{P}=(G_0\setminus 1)\cup R\cup R\inv\cup S\cup S\inv,\]
where we say that a word is \emph{reduced} if there are no subwords of
the form $g_1g_2$, $r^\epsilon r^{-\epsilon}$, $s^\epsilon
s^{-\epsilon}$, $ts_1^{\pm 1}\cdots s_n^{\pm 1}t$ or $s^\epsilon
ts^{-\epsilon}$ for $g_i\in G_0\setminus 1$, $r\in R$, $s,s_i\in S$,
$\epsilon\in\{1,-1\}$. It is easy to see that two reduced words
represent the same element of $G$ if and only if they can be
transformed into each other by swapping adjacent letters $t$ and
$s^\epsilon$.

If $w=p_1\cdots p_n$ is a word and $x$ and element $X$ we say that
$xw$ is \emph{defined} if for all initial segments of $w$ the action
on $x$ is defined, i.e.\ all $xp_1$, $(xp_1)p_2$,\ldots, $(\ldots
(xp_1)\ldots )p_n$ are defined and we set $xw=(\ldots (xp_1)\ldots
)p_n$. Notice that for elements from $G_0$ the action on $X$ is
defined everywhere. If $xw$ is defined and $w'$ is a reduced word
which represents the same element of $G$ as $w$, then $xw'$ is also
defined and we have $xw=xw'$. Thus the expression $xg=y$ makes sense
for $g\in G, x,y\in X$. Furthermore $X$ becomes a gruppoid with
$\hom(x,y)=\{g\in G\mid xg=y\}$ under the natural map
$\hom(x,y)\times\hom(y,z)\to \hom(x,z)$.

If $G$ acts partially on $X$, then there is a canonical partial action
on the set of \emph{pairs}
\[(X)^2 = \{(x,y)\in X^2\mid x\not=y\}.\]

Notice that since $t$ does not fix a point, we have $(x,xt)\in (X)^2$
for all $x\in X$. For $a=(x,y)$ we denote by $\overline a$ the
\emph{flip} $(y,x)$ of $a$. If $ag$ is defined, then so is $\overline
ag=\overline{ag}$.

\begin{definition}
  We call a partial action of $G$ on $X$ \emph{good} if for all pairs
  $a\in (X)^2$ and $g\in G$ the following holds:
 \begin{enumerate}
 \item $ag=a$ implies $g=1$.
 \item If $ag=\overline a$, then $g$ is conjugate to $t$.
 \item $t$ does not fix an element of $X$.
 \end{enumerate}
\end{definition}
Consider the action of $G_0$ on $X$ as a partial action of $G$ on $X$. Then our assumptions on $G_0$ in Theorem~\ref{t:main}
translate exactly into saying that $G$ acts well on $X$.

A word in $\mathcal{P}$ is \emph{cyclically} reduced if every cyclic
permutation of $w$ is reduced. If a word is cyclically reduced, then
every reduced word which represents the same element of $G$ is also
cyclically reduced. Thus, to be cyclically reduced is a property of
elements of $G$. Clearly every element of $G$ is conjugate to a
cyclically reduced one. This shows that in the definition of a
good partial action we can restrict ourselves to cyclically reduced
elements. Note that the cyclically reduced conjugates of $t$ are
the involutions of $G_0$.

\begin{lemma}[Extending $s$]
  Assume that $G$ acts well on $X$ and that for some $x\in X, s\in S$
  and $\epsilon\in\{1,-1\}$ the expression $xs^\epsilon$ is not
  defined (and hence neither is $xts^\epsilon$). Let
  $x'G_0=\{x'g_0\mid g_0\in G_0\}$ be a set of new elements on which
  $G_0$ acts regularly and extend the partial operation of $G$ to
  $X'=X\cup x'G_0$ by putting $xs^\epsilon=x'$ and
  $(xt)s^\epsilon=x't$. Then $G$ acts well on $X'$.
\end{lemma}
\begin{proof}
  Assume $\epsilon=1$, the other case being entirely similar. Let $w$
  be cyclically reduced and $aw=a$ in $X'$. Then the word $w$
  describes a cycle in $(X')^2$ containing $a$. If the cycle contains
  pairs from $X$ only, we are done. If there are two neighbouring
  pairs in the cycle which do not belong to $X$, they must be
  connected by an element $g_0\in G_0\setminus 1$. Thus the cycle
  contains a segment $b,c'_1,d$ or a segment $b,c'_1,c'_2,d$ where
  $b,d\in X$ and $c'_i\notin X$. In the first case we have $bs=c'_1$,
  $c'_1s^{-1}=d$ and in the second case $bs=c'_1$, $c'_1t=c'_2$,
  $c_2s^{-1}=d$. In the first case a cyclic permutation of $w$
  contains the subword $s\cdot s^{-1}$, in the second case $s\cdot
  t\cdot s^{-1}$. Thus $w$ is not cyclically reduced, a contradiction.

  The proof for $aw=\overline a$ is similar: instead of a cycle such
  an element $w$ describes a Moebius strip and we have the additional
  possibility that $a=(x',x'i)$ and $w=i$ for an involution $i\in G$.
\end{proof}
\begin{lemma}[Extending $r$]
  Assume that $G$ acts well on $X$ and that for some $x\in X, r\in R$
  and $\epsilon\in\{1,-1\}$ the expression $xr^\epsilon$ is not
  defined. Choose a set $x'G_0=\{x'g_0\mid g_0\in G_0\}$ of new
  elements on which $G_0$ acts regularly. Extending the partial
  operation of $G$ on $X'=X\cup x'G_0$ by putting $xr^\epsilon=x'$
  yields again a good action of $G$ on $X'$.
\end{lemma}
\begin{proof}
  Consider a non-trivial cycle (or Moebius strip) in $(X)^2$ described
  by a cyclically reduced word $w$. It is easy to see that the cycle
  (Moebius strip) must either be completely contained in $(x'G_0)^2$
  or completely contained in $(X)^2$. In the first case we have a
  Moebius strip of the form $(x',x'i)i=(x'i,x')$ for an involution
  $i\in G_0$. The second case cannot occur since $G$ acts well on $X$
  by assumption.
\end{proof}

\begin{lemma}[Joining $t$-pairs]
  Assume that $G$ acts well on $X$ and let $a=(x,xt)$ and $b=(y,yt)$
  be pairs for which there is no $g\in G$ with $ag=b$. Let $s\in S$ be
  an element which does not yet act anywhere. Extend the action by
  setting $as=b$. Then this action of $G$ on $X$ is again good.
\end{lemma}
\begin{proof}
 Let $w$ be a cyclically reduced word with $cw=c$ for some pair $c\in
 (X)^2$. If $s$ does not occur in $w$, then we have $w=1$ since the
 previous action on $X$ was good. Hence we may assume that $w$
 contains $s$. By cyclically permuting $w$ and taking inverses we may
 also assume that $w=s\cdot w'$ and $aw=a$ and thus $bw'=a$. By
 assumption on $a,b$ the subword $w'$ must contain $s$. Hence we may
 write $w'=u\cdot s^\epsilon v$ for some subword $u$ not containing
 $s$. We distinguish two cases:
 \begin{enumerate}
 \item $\epsilon=1$. Then we must have $bu=a$ or $bu=\overline a$ as
   $s$ is only defined on these pairs. Since $bu=\overline a$ implies
   $b(ut)=a$ both cases contradict the assumption on $a,b$.
 \item $\epsilon=-1$. Then we have $bu=b$ or $bu=\overline b$. If
   $bu=b$, then $u=1$ and $w$ is not reduced. If $bu=\overline b=bt$,
   then $u=t$ and $w$ contains the subword $s\cdot t\cdot s^{-1}$,
   contradicting the assumption that $w$ be reduced.
 \end{enumerate}
 Next we assume that $w$ is cyclically reduced with $cw=\overline c$
 for some pair $c\in (X)^2$. If $w$ does not contain $s$, then $w$ is
 conjugate to $t$ since the previous action on $X$ was good. So we may
 assume that $w=s\cdot w'$ and $aw=\overline a$, i.e. $bw'=\overline
 a$. By choice of $a,b$ we must have $w'$ containing $s$ and we see as
 before that this is impossible.
\end{proof}

\begin{lemma}[Joining other pairs]\label{L:jop}
  Assume that $G$ acts well on $X$ and let $a$ and $b$ be pairs in
  $(X)^2$ such that there is no $g\in G$ with $ag=b$ or $ag=\overline
  b$. Assume furthermore that there is no $g$ in $G$ flipping $b$ and
  that the action of $r\in R$ is not yet defined anywhere. Extending
  the partial action by setting $ar=b$ yields again a good action of
  $G$ on $X$.
\end{lemma}
\noindent Note that $a$ may or may not be a $t$-pair.

\begin{proof}
  Let $w$ be cyclically reduced and $cw=c$ for some pair $c\in (X)^2$.
  If $r$ does not appear in $w$, then we have $w=1$ since the previous
  action on $X$ is good. Hence we may assume again as before that we
  have $w=r\cdot w'$ and $aw=a$. Hence $bw'=a$. By assumption on
  $a,b$, the word $w'$ must contain $r$. Write $w'=u\cdot r^\epsilon
  v$ for some subword $u$ not containing $r$. We distinguish two cases
  \begin{enumerate}
  \item $\epsilon=1$. Then $bu=a$ or $bu=\overline a$ as $r$ is only
    defined there. But this contradicts our choice of $a,b\in (X)^2$.
  \item $\epsilon=-1$. Then we have $bu=b$ or $bu=\overline b$. If
    $bu=b$, then we have $u=1$ by assumption on the previous action
    and $w$ is not reduced. Hence $bu=\overline b$, contradicting the
    assumption that no element of $G$ flips $b$.
  \end{enumerate}
  Now assume that there is some pair $c$ with $cw=\overline c$. If $w$
  does not contain $r$, then $w$ is conjugate to $t$ since the
  previous action is good. Hence we may again assume that we have
  $w=r\cdot w'$ and $aw=\overline a$, hence $bw'=\overline a$. By
  assumption on $a$ and $b$, the word $w'$ must contain $r$ and as
  before we see that this is impossible.
\end{proof}

\begin{corollary}
  Assume that $G$ acts well on $X$ with $|X|\leq \max\{\aleph_0,|G|\}$
  and there are sufficiently many elements of $R$ and $S$ whose action
  is not yet defined anywhere. Then we can extend the partial action
  of $G$ on $X$ to a sharply $2$-transitive action on some appropriate
  superset $Y$.
\end{corollary}
\begin{proof}
  Fix a $t$-pair $a$ in $X_0$. Using the previous lemmas we find the
  set $Y$ with a $2$-sharp action of $G$ on $Y$ with the following
  properties:
  \begin{enumerate}
  \item all $t$-pairs are connected to $a$;
  \item any pair can be flipped by an element of $G$.
  \end{enumerate}
  The last property can be achieved using Lemma \ref{L:jop}: if $b$
  cannot be flipped at some stage of the construction, we can connect
  $a$ and $b$ in a later stage. Then $b$ can be flipped as $a$ can.
  Note that Lemma \ref{L:jop} is used only for $t$-pairs $a$.

  This easily implies that the action of $G$ on $Y$ is sharply
  $2$-transitive: It is left to show that all pairs are to connected
  $a$. Let $b$ be a pair and $g\in G$ so that $bg=\overline b$. Then
  $g=hgh\inv$ for some $h\in G$. This implies $(bh)t=\overline{bh}$,
  so $bh$ is a $t$-pair and whence connected to $a$
\end{proof}

\noindent This concludes the proof of Theorem \ref{t:main} and its
corollary. Note that our construction yields a group action for which
no involution has a fixed point.



\vspace{2cm}

\begin{small}
\noindent\parbox[t]{15em}{
Katrin Tent,\\
Mathematisches Institut,\\
Universit\"at M\"unster,\\
Einsteinstrasse 62,\\
D-48149 M\"unster,\\
Germany,\\
{\tt tent@math.uni-muenster.de}}
\hfill\parbox[t]{18em}{
Martin Ziegler,\\
Mathematisches Institut,\\
Albert-Ludwigs-Universit\"at Freiburg,\\
Eckerstr. 1,\\
D-79104 Freiburg,\\
Germany,\\
{\tt ziegler@uni-freiburg.de}}

\end{small}
\end{document}